\documentclass[11pt,reqno]{amsart}

\numberwithin{equation}{section}
\usepackage{amssymb,amsmath}
\usepackage{amsmath,amsfonts,amsthm,mathrsfs}
\usepackage{color,graphics,epsfig}

\newcommand{\calL}{\mathcal{L}}

\newcommand{\mA}{\mathbb{A}}

\newcommand{\mC}{\mathbb{C}}

\newcommand{\mR}{\mathbb{R}}

\newcommand{\mT}{\mathbb{T}}

\newcommand{\mZ}{\mathbb{Z}}

\newtheorem{theorem}{Theorem}[section]
\newtheorem{lemma}[theorem]{Lemma}
\newtheorem{corollary}[theorem]{Corollary}
\newtheorem{proposition}[theorem]{Proposition}

\newtheorem{claim}[theorem]{Claim}

\theoremstyle{definition}

\newtheorem{example}[theorem]{Example}

\theoremstyle{definition}
\newtheorem{definition}[theorem]{Definition}

\theoremstyle{definition}
\newtheorem{notation}[theorem]{Notation}

\begin{document}

\keywords{Riccati equations, Banach algebras, Systems over rings, 
Optimal control, Spatially distributed dynamical systems}

\subjclass{Primary 46J05; Secondary 93D15, 58C15, 47N20}

\title[Riccati equations in Banach algebras]{On Riccati equations in Banach algebras}

\author{Ruth Curtain} \address{Department of Mathematics, University
  of Groningen, P.O. Box 800, 
9700~AV Groningen, The Netherlands.}
\email{R.F.Curtain@math.rug.nl}

\author{Amol Sasane}
\address{Department of Mathematics, Royal Institute of Technology,
    Stockholm 100 44, Sweden.}
\email{sasane@math.kth.se}

\begin{abstract}
Let $R$ be a commutative complex Banach algebra with the involution 
$\cdot ^\star$ and suppose that $A\in R^{n\times n}$, $B\in R^{n\times
  m}$, $C\in R^{p\times n}$.  The question of when the Riccati
equation
$$
PBB^\star P-PA-A^\star P-C^\star C=0
$$
has a solution $P\in R^{n\times n}$ is investigated. A counterexample
to a previous result in the literature on this subject is given,
followed by sufficient conditions on the data guaranteeing the
existence of such a $P$. Finally, applications to spatially
distributed systems are discussed.
\end{abstract}

\maketitle

\section{Introduction}

If $A\in \mC^{n\times n}$, $B\in \mC^{n\times m}$ and $C\in
\mC^{p\times m}$, then the {\em Riccati equation} is
$$
PBB^*P-PA-A^*P-C^*C=0
$$
in the unknown $P\in \mC^{n\times n}$. This is a fundamental equation
associated with the problem of optimal control of linear control
systems with a quadratic cost, and the following is a well known
result about the existence of solutions.

\begin{proposition}
\label{prop1}
  Let $A\in \mC^{n\times n}$, $B\in \mC^{n\times m}$, $C\in
  \mC^{p\times n}$ be such that $(A,B)$ is stabilizable and $(A,C)$ is
  detectable. Then there is a unique positive semidefinite solution
  $P\in \mC^{n\times n}$ to the Riccati equation
$$
PBB^*P-PA-A^*P-C^*C=0,
$$
such that $A-BBP^*$ is exponentially stable, that is, $\textrm{\em
  Re}(\lambda)<0$ for all eigenvalues $\lambda$ of $ A-BBP^*$.
\end{proposition}

(Recall that the pair $(A,B)$ is {\em stabilizable} if there exists a
$F \in \mC^{m\times n}$ such that $A+BF$ is asymptotically stable, and
the pair $(A,C)$ is {\em detectable} if the pair $(A^*,C^*)$ is
stabilizable.)

There has been old (see \cite{Byr}) and recent (see \cite{Cur})
renewed interest in the following question: if the data $A,B,C$ have
entries in a Banach algebra, then does there exist a solution $P$ also
with entries from the same Banach algebra? In this article, we
investigate this question. We begin by fixing some notation.

\begin{notation}
  Throughout the article, $R$ will denote a commutative, unital,
  complex, semisimple Banach algebra, which possesses an involution
  $\cdot^\star$. 

  On the other hand, the usual adjoint of a matrix $M=[m_{ij}]\in
  \mC^{p\times m}$ will be denoted by $M^* \in \mC^{m\times p}$, that
  is, $ M^* =[\overline{m_{ji}}]$.

$M(R)$ will denote the maximal ideal space of $R$, equipped with the
weak-$\ast$ topology. For $x\in R$, we will denote its Gelfand
transform by $\widehat{x}$, that is, 
$$
\widehat{x}(\varphi)=\varphi(x), \quad \varphi \in M(R),\; x\in R.
$$
For a matrix $M\in R^{p\times m}$, whose entry in the $i$th row and
$j$th column is denoted by $m_{ij}$, we define $M^\star\in R^{m\times
  p}$ to be the matrix whose entry in the $i$th row and $j$th column
is $m_{ji}^\star$. Also by $\widehat{M}$ we mean the $p\times m$
matrix, whose entry in the $i$th row and $j$th column is the
continuous function $\widehat{m_{ij}}$ on $M(R)$. Summarizing, if
$M=[m_{ij}]\in R^{p\times m}$, then
\begin{eqnarray*}
M^\star 
&=&  
\Big[m_{ji}^\star \Big]\in R^{m\times p},
\\
\widehat{M} 
&=& 
\Big[\widehat{m_{ij}}\Big]\in \Big(C(M(R);\mC)\Big)^{p\times m},
\\
\Big(\widehat{M}(\varphi)\Big)^* 
&=&
\Big[\overline{\widehat{m_{ji}}(\varphi)}\Big]\in \mC^{p\times m}.
\end{eqnarray*}
\end{notation}

The following claim was made in \cite[Theorem~2.2, p.248]{Byr}.

\begin{claim}
\label{claim_chris} 
Let $A\in R^{n\times n}$, $B\in R^{n\times m}$, $C\in R^{p\times n}$ 
be such that for all $\varphi\in M(R)$,
$(\widehat{A}(\varphi),\widehat{B}(\varphi))$ and
$(\widehat{A^\star}(\varphi),\widehat{C^\star}(\varphi))$ are
controllable.  Then there exists a solution $P\in R^{n\times n}$ such
that
\begin{align}
\label{riccati}
PBB^\star P-PA-A^\star P-C^\star C=0.
\end{align} 
\end{claim}

(Recall that for matrices $A\in \mC^{n\times n}$ and $B\in
\mC^{n\times m}$, the pair $(A,B)$ is said to be {\em controllable} if
$\textrm{rank}\left[\begin{array}{ccccc} B & AB & A^2B & \dots &
    A^{n-1}B\end{array}\right]=n$.)

However, in Section~\ref{section_counterexample}, we will see a
counterexample to Claim~\ref{claim_chris}, showing that this is not
true in general, without invoking extra assumptions, and this is our
main result:

\begin{theorem}
\label{main_theorem}
  Let $A\in R^{n\times n}$, $B\in R^{n\times m}$, $C\in R^{p\times n}$
  satisfy the following: for all $\varphi\in M(R)$, 
\begin{itemize}
\item[(A1)] $\widehat{(A^\star)}(\varphi)= (\widehat{A}(\varphi))^*$,
\item[(A2)] $\widehat{(BB^\star)}(\varphi)= 
            \widehat{B}(\varphi) (\widehat{B}(\varphi))^*$,
\item[(A3)] $\widehat{(C^\star C)}(\varphi)= 
            (\widehat{C}(\varphi))^* \widehat{C}(\varphi)$,
\item[(A4)] $(\widehat{A}(\varphi),\widehat{B}(\varphi))$ is
            stabilizable,
\item[(A5)] $(\widehat{A}(\varphi),\widehat{C}(\varphi))$ is
            detectable.
\end{itemize}
Then there exists a $P\in R^{n\times n}$ such that 
\begin{enumerate}
\item $PBB^\star P-PA-A^\star P-C^\star C=0$, 
\item $A-BB^\star P$ is exponentially stable, and
\item for all $\varphi\in M(R)$, $\widehat{P}(\varphi)$ is positive
  semidefinite.
\end{enumerate}
\end{theorem}

In the following we define what is meant by ``exponentially stable''.

\begin{definition}
Let $R$ be a commutative, unital, complex, semisimple Banach 
algebra. If $A\in R^{n\times n}$, let $M_A:R^n\rightarrow R^n $ be the
multiplication map by the matrix $A$, that is, $v\mapsto Av$ ($v\in
R^n$).  Then $R^{n\times n}$ is a unital complex Banach algebra (for
example) with the norm
$$
\|A\|:=\|M_A\|_{\calL(R^n)}\quad (A\in R^{n\times n})
$$
where $\calL(R^n)$ denotes the set of all continuous linear
transformations from $R^n$ to $R^n$, and $R^n$ is the Banach space
equipped (for example) with the norm
$$
\|x\|=\max\{\|x_k\|:1\leq k\leq n\} \textrm{ for } 
x=\left[\begin{array}{c} x_1\\\vdots\\ x_n\end{array}\right],
$$
and $\calL(R^n)$ is equipped with the usual operator norm:
$$
\|M_A\|_{\calL(R^n)}=\sup\{ \|Av\|:v\in R^n \textrm{ with }\|v\|\leq 1\}.
$$
For $A\in R^{n\times n}$, we define
$$
e^{A}=\sum_{k=0}^\infty \frac{1}{k!}A^k.
$$
The absolute convergence of this series is established just as in the
scalar case.

$A\in R^{n\times n}$ is said to be {\em exponentially stable} if there
exist positive constants $C$ and $\epsilon$ such that 
$$
\|e^{tA}\|\leq Ce^{-\epsilon t} \textrm{ for all } t \geq 0.
$$
\end{definition}

\begin{lemma}
\label{lemma_spectra}
Let $A\in R^{n\times n}$. Then the following are equivalent:
\begin{enumerate}
\item $\lambda$ belongs to the spectrum of $A\in R^{n\times n}$.
\item $\lambda$ belongs to the spectrum of $M_A\in \calL(R^n)$.
\item $\lambda$ belongs to the spectrum of $\widehat{A}(\varphi)$ for some $\varphi \in M(R)$. 
\end{enumerate}
\end{lemma}
\begin{proof}
  The equivalence of (1) and (3) follows from the fact
  \cite[Theorem~8.1, p.830]{GohGolKaa} that $A\in R^{n\times n}$ is
  invertible in the Banach algebra $R^{n\times n}$ if and only if the
  matrix $\widehat{R}(\varphi)$ is invertible in $\mC^{n\times n}$ for
  each $\varphi \in M(R)$.  For $A\in R^{n\times n}$, it can be seen
  that $A$ is invertible in $R^{n\times n}$ if and only if $M_A$ is
  invertible in $\calL(R^n)$. Indeed, the `only if' part is
  trivial, since if $AA^{-1}=I=A^{-1}B$, then $M_AM_{A^{-1}}=
  I=M_{A^{-1}}M_A$. Vice versa, if $M_A T=I=TM_A$ for some $T\in
  \calL(R^n)$, then set
$$
A^{-1}_:= \left[\begin{array}{ccc} Te_1 & \dots & Te_n
  \end{array}\right]\in R^{n\times n}, 
$$
where $e_k$ ($k=1,\dots, n$) denotes the vector in $R^n$ with $1$ in
the $k$th position and zeros elsewhere. Then $M_{A^{-1}}e_k=Te_k$ for
all $k=1,\dots, n$, and so $M_{A^{-1}}v=Tv$ for all $v\in R^n$. Hence
(1) and (2) are equivalent.
\end{proof}

The following gives a characterization of exponential stability.

\begin{proposition}
\label{prop_char_exp_stab}
Let $A\in R^{n\times n}$. Then $A$ is exponentially stable if and only if
$$
\sup \{ \textrm{\em Re}(\lambda):\lambda \textrm{ is an eigenvalue of }
\widehat{A}(\varphi) \textrm{ for some } \varphi \in M(R)\} <0.
$$
\end{proposition}
\begin{proof} We recall the result in semigroup theory that the
  semigroup generated by a continuous linear transformation on a
  Banach space is exponentially stable if and only if the supremum of
  the real parts of points in the spectrum of the operator is strictly
  negative \cite[Corollary~IV.2.4, p.252 and Proposition~V.1.7,
  p.299]{EngNag}. Using this, we see that $A$ is exponentially stable
  if and only if $ \sup \{ \textrm{Re}(\lambda):\lambda \textrm{
    belongs to the spectrum of }M_A \in \calL(R^n)\}$ is negative.
  The proof is now finished by using Lemma~\ref{lemma_spectra}.
\end{proof}

The proof in Section \ref{Proof} of our main result above is similar
to the approach in \cite{Byr}, where we first take Gelfand transform
of our equation, and show that the pointwise solution is continuous.
Then we use the Banach algebra operational calculus to ensure that
this continuous solution is actually the Gelfand transform of a matrix
with entries from the Banach algebra.

In Section~\ref{Applications} we discuss the applications of this
result to the control of spatially invariant systems.


\section{Counterexample to Claim~\ref{claim_chris}}
\label{section_counterexample}

\begin{example}
\label{example_counterexample}
Consider the Banach algebra $C^{1}(\mT)$ of all continuously
differentiable functions on the unit circle with pointwise operations
and the norm
$$
\|f\|_{C^{1}(\mT)}=\|f\|_\infty+\|f'\|_\infty,
$$
with the understanding that $f'(e^{i\theta}):=
\frac{dF}{d\theta}(\theta)$, where $F(\theta):=f(e^{i\theta})$
($\theta\in \mR$). Then $C^1(\mT)$ is a semisimple commutative unital
complex Banach algebra. Every point $z$ on $\mT$ gives rise to the
complex homomorphism
$$
f\mapsto f(z): C^{1}(\mT)\rightarrow \mC.
$$
Also, all complex homomorphisms arise in this manner, and this can be
seen as follows. If $\varphi$ is a complex homomorphism which is not a
point evaluation at any point of $\mT$, then given any $z\in \mT$,
there is a corresponding $f\in C^{1}(\mT)$ such that $\varphi(f)=0$,
but $f(z)\neq 0$. So in fact for $w$'s belonging to a small
neighbourhood of this $z$, we have $|f(w)|\geq \delta_z>0$. But $\mT$
is compact, and so there exist finitely many functions $f_1,\dots,
f_n$ such that $\varphi(f_i)=0$ for each $i$, and for every point on
the unit circle, at least one of the functions $f_i$ ($1\leq i\leq n$)
is nonzero there. Thus the function
$$
g:=f_1 \overline{f_1}+\dots +f_n \overline{f_n}
$$
is in $C^1(\mT)$, it satisfies $\varphi(f)=0$, and is nonzero on
$\mT$. But being nonzero on $\mT$, $g$ is invertible as an element of
$C^1(\mT)$, a contradiction to the fact that a maximal ideal cannot
contain units.  Hence the maximal ideal space of $C^1(\mT)$ can be
identified with the unit circle $\mT$.

The Banach algebra $C^1(\mT)$ possesses the involution $\cdot^\star$
defined by
$$
f^\star(z)= \overline{f(\overline{z})} \;\; (z\in \mT), 
\textrm{ for }  f\in C^1(\mT). 
$$
Now consider the following $A,B,C\in C^1(\mT)$:
$$
A(z)=z, \quad B(z)=1, \quad C(z)=1 , \quad (z\in \mT).
$$
Then we have that $A=A^\star$, $B=B^\star=C=C^\star$.

For each $z\in \mT$, $(\widehat{A}(\varphi),\widehat{B}(\varphi))
=(z,1)=(\widehat{A^\star}(\varphi),\widehat{C^\star}(\varphi))$ is
controllable. Thus all the hypotheses of the Claim~\ref{claim_chris}
are satisfied. But we will show below that the corresponding Riccati
equation has no solution in the Banach algebra $C^1(\mT)$.

The Riccati equation is $ P^2-2zP-1=0$. Let us suppose that this has a
solution $P\in C^1(\mT)$. Then we obtain
$$
(\widehat{P}(z))^2-2z \widehat{P}(z)-1=0\quad (z\in \mT),
$$
that is $(\widehat{P}(z)-z)^2=z^2+1$ ($z\in \mT$).  We will now show
the following:

\medskip 

\noindent {\bf Claim:} There is no $Q\in C^1(\mT)$ such that 
$(\widehat{Q}(z))^2=z^2+1$ ($z\in \mT$).

\medskip 

\noindent It is not hard to see that the function $g$ given by 
$$
g(e^{i\theta}):= \left\{\begin{array}{ll}
    \sqrt{2} \sqrt{\cos \theta}\;\!e^{i\frac{\theta}{2}} 
    & \textrm{if } \cos \theta\geq 0,\\
    \sqrt{2} \sqrt{-\cos \theta}\;\!e^{i\frac{(\theta+\pi)}{2}}&
    \textrm{if } \cos \theta<0\end{array}\right.
$$ 
does satisfy $(g(z))^2=z^2+1$ ($z\in \mT$). As the function $g$ is not
differentiable when $\theta=\frac{\pi}{2}$, it follows that $g\not\in
C^1(\mT)$.  Since $g$ has two roots on $\mT$, namely at $i$ and at
$-i$, it follows from $Q^2=g^2$, that $Q$ is either $g$ or $-g$ or
$hg$ or $-hg$, where
$$
h(z):=\left\{\begin{array}{ll}
    1 & \textrm{if Re}(z)\geq 0,\\
    -1 & \textrm{if Re}(z)<0.
\end{array}\right. 
$$
But none of these functions is differentiable when
$\theta=\frac{\pi}{2}$.  This completes the proof of the fact that
there is no $Q\in C^1(\mT)$ such that $ (\widehat{Q}(z))^2=z^2+1$
($z\in \mT$).

So we conclude that the Claim~\ref{claim_chris} is false.
\hfill$\Diamond$
\end{example}

\section{Proof of the main result}\label{Proof}

We will need the following two results. The first one says that if we
consider the classical Riccati equation with constant complex
matricial data $A,B,C$, then the solution $P$ depends continuously on
the $A,B,C$; see \cite[Theorem~1.2.1, p.260]{LanRod}.

\begin{proposition}
\label{prop2}
  With the same notation as in Proposition~\ref{prop1}, the
  maximal Hermitian solution $P(A,B,C)$ of the Riccati equation is a
  continuous function of $(A,B,C)$. (Here $P(A,B,C)$ is viewed as a
  function on a subset of $\mC^{n^2}\times \mC^{nm}\times \mC^{pn}$
  with the usual topology).
\end{proposition}

The next result we will need is the following (see \cite[p.155]{Hay}),
and this will be used to pass from continuous functions on $M(R)$ to
elements of $R$.

\begin{proposition}
\label{prop3}
Let $h_1,\dots, h_s$ be continuous functions on $M(R)$. Suppose that
$f_1,\dots f_\ell$ in $R$ and $
G_1(z_1,\dots, z_{s+\ell}),\dots, G_t(z_1,\dots,z_{s+\ell})
$ are holomorphic functions 
with $t\geq s$ defined on a neighbourhood of the joint spectrum
$$
\sigma(h_1,\dots, h_s, f_1,\dots, f_\ell)
:=
\{ (h_1(\varphi),\dots,h_s(\varphi), \widehat{f_1}(\varphi),\dots,
\widehat{f_\ell}(\varphi)):\varphi\in M(R)\},
$$
such that 
\begin{equation}
\label{eq_hayashi_cont_soln}
G_k(h_1,\dots, h_s, \widehat{f_1},\dots,\widehat{f_\ell})=0 
\textrm{ on }M(R) \textrm{ for }1\leq k\leq t.
\end{equation}
If the rank of the Jacobi matrix $\displaystyle
\frac{\partial(G_1,\dots, G_t)}{\partial(z_1,\dots, z_s)}$ is $s$ on
$\sigma(h_1,\dots, h_s, f_1,\dots, f_\ell)$, then there exist elements
$g_1,\dots, g_s$ in $R$ such that
$$
\widehat{g_1}=h_1,\dots,\widehat{g_s}=h_s .
$$
\end{proposition}

We are now ready to prove our main result.

\begin{proof}[Proof of Theorem~\ref{main_theorem}] If we fix a
  $\varphi\in M(R)$, then owing to the assumptions (A4) and (A5), we
  know that there is a unique solution in $\mC^{n\times n }$, which we
  will denote by $\Pi(\varphi)$, such that it is positive
  semidefinite, 
\begin{equation}
\label{riccati_1}
\Pi(\varphi)\widehat{B}(\varphi)(\widehat{B}(\varphi))^*\Pi(\varphi)
-\Pi(\varphi)\widehat{A}(\varphi)-(\widehat{A}(\varphi))^*\Pi(\varphi)
-(\widehat{C}(\varphi))^*\widehat{C}(\varphi)=0,
\end{equation}
and $\widehat{A}(\varphi)-\widehat{B}(\varphi)(\widehat{B}(\varphi))^*
\Pi(\varphi)$ is exponentially stable.

Moreover, from Proposition~\ref{prop2}, it follows that the map
$\varphi\mapsto \Pi(\varphi)$ is continuous on $M(R)$.

Finally we will apply Proposition~\ref{prop3}. We have in our case
$s=n^2$, $t=n^2$, the $h_i$'s are the components of $\Pi$ and the
$f_i$'s are the components of $A,A^\star,BB^\star,C^\star C$ (which
are totally $\ell=n^2+n^2+n^2+n^2=4n^2$ in number).  The maps
$G_1,\dots G_{t=n^2}$ are the $n^2$ components of the map
$$
(\Theta, U,V,W,X)\mapsto \Theta W \Theta-\Theta U-V\Theta-X.
$$
(In the above we have the replacements of $A, A^*,BB^*,CC^*$ by the
complex variables which are the components of $U,V,W,X$, respectively.
The replacements of the $P$ in the Riccati equation is by the complex
variables which are the components of $\Theta$.)  Clearly the above
map is holomorphic not just on the joint spectrum, but rather in the
whole of $\mC^{s+\ell}=\mC^{n^2+ 4n^2}$. 

In light of the assumptions (A1)-(A3) and and \eqref{riccati_1}, we
know that there is a continuous solution $\Pi$ on the maximal ideal
space such that for all $k$,
\begin{equation}
\label{eq_riccati_2}
G_k(\Pi, \widehat{A},\widehat{A^\star}, \widehat{BB^\star},\widehat{C^\star C})=0 
\end{equation}
on $M(R)$ (that is, condition \eqref{eq_hayashi_cont_soln} in
Proposition~\ref{prop3} is satisfied).

So we now investigate the Jacobian with respect to the variables in
$\Theta$. The Jacobian with respect to the $\Theta$ variables at the
point
$$
\!\!\!\Big(\!\Pi(\varphi), \;\;\widehat{A}(\varphi), \;\;
(\widehat{A}(\varphi))^*,\;\;\widehat{BB^\star}(\varphi)\!
=\!\widehat{B}(\varphi)(\widehat{B}(\varphi))^*,\;\;
\widehat{C^\star C}(\varphi)\!\!=
(\widehat{C}(\varphi))^*\widehat{C}(\varphi)
\!\Big)
$$ 
is the following linear transformation $\Lambda$ from
$\mC^{n^2}\rightarrow \mC^{n^2}$:
$$
\Theta \mapsto \Theta
\widehat{B}(\varphi)(\widehat{B}(\varphi))^*\Pi(\varphi)+ 
\Pi(\varphi)
\widehat{B}(\varphi)(\widehat{B}(\varphi))^*\Theta
-\Theta\widehat{A}(\varphi)-(\widehat{A}(\varphi))^*\Theta,
$$
that is, 
$$
\!\Theta \mapsto -\Bigg(
\!\!\Big(\widehat{A}(\varphi)-
\!\widehat{B}(\varphi)(\widehat{B}(\varphi))^*\Pi(\varphi)\Big)^*\Theta
+\Theta
\Big(\widehat{A}(\varphi)-
\!\widehat{B}(\varphi)(\widehat{B}(\varphi))^*\Pi(\varphi)\!\Big)\!\!
\Bigg)
$$
The set of eigenvalues of $\Lambda$ consists of the numbers 
$$
-(\overline{\lambda}+\mu),
$$
where $\lambda$, $\mu$ belong to the set of eigenvalues of
$\widehat{A}(\varphi)-\widehat{B}(\varphi)(\widehat{B}(\varphi))^*
\Pi(\varphi)$; see for example \cite[Proposition~7.2.3]{Ber}.  But
since $\widehat{A}(\varphi)-\widehat{B}(\varphi)
(\widehat{B}(\varphi))^* \Pi(\varphi)$ is exponentially stable, all
its eigenvalues have a negative real part.  Hence
$-(\overline{\lambda}+\mu)\neq 0$ for all $\lambda$, $\mu$ belonging
to the set of eigenvalues of the matrix 
$\widehat{A}(\varphi)-\widehat{B}(\varphi)(\widehat{B}(\varphi))^*
\Pi(\varphi)$. Consequently, the map $\Lambda$ is invertible from
$\mC^{n^2}$ to $\mC^{n^2}$, and its rank is $n^2=s$. So by
Proposition~\ref{prop3}, there exists a $P\in R^{n\times n}$ such that
$\widehat{P}(\varphi)=\Pi(\varphi)$ for all $\varphi \in M(R)$. From
\eqref{eq_riccati_2}, it follows (using the fact that $R$ is
semisimple) that
$$
PBB^\star P-PA-A^\star P-C^\star C=0.
$$
From the property possessed by the pointwise solutions $\Pi(\varphi)$
($\varphi \in M(R)$) of the constant complex matricial Riccati
equations \eqref{riccati_1}, we have that for all $\varphi\in M(R)$,
all eigenvalues of $\widehat{(A-BB^\star P)}(\varphi)$ have a negative
real part. But the set-valued map taking a square complex matrix of
size $n\times n$ to its spectrum (a set of $n$ complex numbers) is a
continuous map; see for example \cite[II,\S 5, Theorem~5.14,
p.118]{Kat}. Since $M(R)$ is compact in the Gelfand topology (the
weak-$\ast$ topology induced on $M(R)$ considered as a subset of
$\calL(R;\mC)$), it follows that
$$
\sup \{ \textrm{Re}(\lambda):\lambda \textrm{ is an eigenvalue of }
\widehat{(A-BB^\star P)}(\varphi) \textrm{ for some } \varphi \in M(R)\} <0.
$$
From Proposition~\ref{prop_char_exp_stab}, it follows that $A-BB^\star
P$ is exponentially stable.  Finally, again by the property possessed
by the pointwise solution $\Pi$, we have that for all $\varphi\in
M(R)$, $\widehat{P}(\varphi)$ is positive semidefinite.  This
completes the proof of Theorem~\ref{main_theorem}.
\end{proof}

\begin{example} [Example~\ref{example_counterexample} revisited]  
Let us check what went wrong with our example considered earlier. 
Recall that the Banach algebra was $C^1(\mT)$, the involution
$\cdot^\star$ was given by
$$
f^\star(z)=\overline{f(\overline{z})}\quad (z\in \mT),
$$
and the Riccati equation data was given by $A=z$, $B=C=1$. We see that
in Theorem~\ref{main_theorem}, for this particular example, although
the assumptions (A2)-(A5) are satisfied, (A1) fails to hold. Indeed,
$$
A^\star(z)=z \quad (z\in \mT)
$$
and so for $z\in \mT\setminus\{-1,1\}$, we have 
$$
\widehat{(A^\star)}(z)=z\neq \overline{z}= (\widehat{A}(z))^*.
$$
So it is no surprise that a solution does not exist to the Riccati
equation.

If instead, we consider the following new $A$, given by 
$$
A(z)=z+\frac{1}{z} \quad (z\in \mT),
$$
then 
$$
\widehat{(A^\star)}(z)=z+\frac{1}{z}=\overline{z+\frac{1}{z}}
=(\widehat{A}(z))^*\quad(z\in
\mT).
$$
With the same $B=C=1$ considered earlier, we see that the assumptions
(A1)-(A5) in Theorem~\ref{main_theorem} are all satisfied now, and the
Riccati equation
$$
PBB^\star P-PA-A^\star P-C^\star C=P^2 -2\Big(z+\frac{1}{z}\Big)P-1=0
$$
has a solution $P\in C^1(\mT)$, given by:
$$
P(e^{i\theta})=2\cos \theta+\sqrt{(2\cos \theta)^2+1} \quad (\theta\in \mR).
$$
Clearly, $P(e^{i\theta})$ is positive semidefinite (as it is $\geq
0$). Moreover, 
$$
\widehat{(A-BB^\star P)}(e^{i\theta})=-\sqrt{(2\cos \theta)^2+1}\leq
-1,
$$
and so $A-BB^\star P$ is exponentially stable by
Proposition~\ref{prop_char_exp_stab}.  \hfill$\Diamond$
\end{example}

We observe that whether or not assumptions (A1)-(A3) in
Theorem~\ref{main_theorem} hold is intimately related to the choice of
the involution $\cdot^\star$ in the Banach algebra $R$. For some
commutative Banach algebras with involutions, this is automatic,
namely if it is {\em symmetric}. 

\begin{definition}
  A unital Banach algebra $R$ with an involution $\cdot^\star$ is said
  to be {\em symmetric} if for every $x\in R$, the spectrum of
  $xx^\star$ (as an element of $R$) is contained in $[0,+\infty)$.
  Equivalently, $R$ is symmetric if and only if for every $x\in R$
  satisfying $x=x^\star$ implies that the spectrum of $x$ is real; see
  \cite[\S2.3, p.2700]{GroLei}.  The involution is then called a {\em
    symmetric involution}.
\end{definition}

In the case when $R$ is commutative, this is equivalent to the
following (see \cite[Definition~2, \S I.8, p.57]{GelRaiShi}).
  
\begin{proposition}
Let $R$ be a commutative unital complex semisimple Banach algebra with an 
involution $\cdot^\star$. Then  the following are equivalent:
\begin{enumerate}
\item $R$ is symmetric.
\item For each $x\in R$, $ \varphi(x^\star)=\overline{\varphi(x)}$
  ($\varphi\in M(R)$).
\item For each $x\in R$, $1+x^\star x$ is invertible in $R$.
\end{enumerate}
\end{proposition} 
\begin{proof} The equivalence of (2) and (3) is precisely
\cite[Theorem~2, p.59]{GelRaiShi}. 

Let us now show that (1) and (2) are equivalent. Suppose that (1)
holds. Let $x\in R$. Then using the fact that the spectrum of an
element is the range of its Gelfand transform, it follows that
$$
\widehat{(1+ x^\star x)}(\varphi)= 1+\widehat{x^\star x}(\varphi)\geq 1,
$$
and in particular, $\widehat{(1+ x^\star x)}(\varphi)\neq 0$ for all
$\varphi\in M(R)$. Thus $1+ x^\star x$ is invertible as an element
of $R$. So (3), and consequently also (2), holds.

Now suppose that (2) holds. Let $x\in R$. We have that 
$$
\widehat{x^\star x}(\varphi)=\varphi(x^\star x) 
=\varphi(x^\star)\varphi(x)=\overline{\varphi(x)}\varphi(x)
=|\varphi(x)|^2\geq 0.
$$
Since the spectrum of $x^\star x$ is the range of its Gelfand
transform, it follows that the spectrum of $x^\star x$ is contained in
the half line $[0,+\infty)$. Thus $R$ is symmetric.
\end{proof}

In particular, all $C^*$-algebras are symmetric.

We have the following consequence of our main result.

\begin{corollary} 
  Let $R$ be a commutative unital complex semisimple symmetric Banach
  algebra with a symmetric involution $\cdot^\star$.  Let $A\in
  R^{n\times n}$, $B\in R^{n\times m}$, $C\in R^{p\times n}$ satisfy
  the following: for all $\varphi\in M(R)$,
\begin{itemize}
\item[(i)] $(\widehat{A}(\varphi),\widehat{B}(\varphi))$ is
            stabilizable,
\item[(ii)] $(\widehat{A}(\varphi),\widehat{C}(\varphi))$ is
            detectable.
\end{itemize}
Then there exists a $P\in R^{n\times n}$ such that 
\begin{enumerate}
\item $PBB^\star P-PA-A^\star P-C^\star C=0$, 
\item $A-BB^\star P$ is exponentially stable, and
\item $P=P^\star$ and the spectrum of $P$ (as an element of the Banach
  algebra $R^{n\times n}$) is contained in $[0,+\infty)$.
\end{enumerate}
\end{corollary}
\begin{proof} This is an immediate consequence of
  Theorem~\ref{main_theorem} since (A1)-(A3) are satisfied
  automatically owing to the symmetry of the Banach algebra $R$. Also
  (3) above follows from the conclusion (3) of
  Theorem~\ref{main_theorem} and the symmetry property of $R$. Indeed
  we have that for all $\varphi\in M(R)$, 
$$
\widehat{P^\star}(\varphi)=(\widehat{P}(\varphi))^*=\widehat{P}(\varphi),
$$
where the first equality follows from the symmetry of $R$ and the
second equality follows from Theorem~\ref{main_theorem}.(3). Thus
$P^\star=P$. The spectrum of $P$ coincides with the set containing the
eigenvalues of $\widehat{P}(\varphi)$ ($\varphi\in M(R)$) and since
for each $\varphi \in M(R)$, $\widehat{P}(\varphi)$ is positive
semidefinite, it follows that the spectrum of $P$ is contained in
$[0,+\infty)$.
\end{proof}

\begin{example}[Example~\ref{example_counterexample} revisited] 
Consider the same Banach algebra $C^1(\mT)$ as in 
Example~\ref{example_counterexample}, and the same Riccati equation
data $A,B,C$ given there, namely $A=z$, $B=C=1$, but now with a new
involution on $C^1(\mT)$, given simply by
$$
f^\star(z)=\overline{f(z)} \quad (z\in \mT).
$$
Now the $A$ does satisfy assumption (A1) from
Theorem~\ref{main_theorem}, since 
$$
\widehat{(A^\star)}(z)=\overline{z}= (\widehat{A}(z))^* \quad (z\in \mT).
$$
Also, as before the assumptions (A2)-(A5) are satisfied. The
corresponding Riccati equation is
$$
PBB^\star P-PA-A^*P-C^\star C=P^2 -(z+\overline{z})P-1=0,
$$
and it has a solution $P\in C^1(\mT)$, given by:
$$
P(e^{i\theta})=\cos \theta+\sqrt{(\cos \theta)^2+1} \quad (\theta\in
\mR).
$$
Clearly, $P(e^{i\theta})$ is positive semidefinite (as it is $\geq
0$). Moreover, from
$$
\widehat{(A-BB^\star P)}(\varphi)=-\sqrt{(\cos  \theta)^2+1}\leq -1,
$$ 
it follows that that $A-BB^\star P$ is exponentially stable.  
\hfill$\Diamond$
\end{example}

\section{Application to spatially invariant
  systems}\label{Applications}

In this section we discuss the applications of our results to control
problems for spatially invariant systems introduced in
\cite{BamPagDah}. The analysis of spatially invariant systems can be
greatly simplified by taking Fourier transforms, see \cite{BamPagDah},
\cite{CurIftZwa}.  This yields systems described by multiplication
operators with symbols $A,B,C\in (L^\infty(\mT))^{n\times n}$. The
Linear Quadratic Regulator (LQR) control design is to use the feedback
$F=-BB^*P$, where $P$ is the bounded, self-adjoint, stabilizing
solution to the LQR Riccati equation \eqref{riccati} on the Hilbert
space $(L^2(\mT))^n$. For the design of implementable controllers it
is important that the gain operator have a spatially decaying property
(see \cite{BamPagDah}).  This translates into the mathematical
question of when the LQR Riccati equation \eqref{riccati} has a
stabilizing solution in a suitable subalgebra (for example,
$(L^1(\mT))^{n\times n}$ is a subalgebra of ${\mathcal
  L}((L^2(\mT))^n)=(L^\infty(\mT))^{n\times n}$).  So the spatially
decaying property now translates into finding suitable subalgebras of
$(L^\infty(\mT))^{n\times n}$, in particular, the weighted Wiener
algebras. From our results in the previous sections it suffices to
identify the symmetric Wiener algebras for the case $n=1$.  In the
following example we show that a large class of Wiener subalgebras of
$L^\infty(\mT)$ do have this property.

\begin{example}[Even-weighted Wiener algebras]
\label{symex} 
Let $\alpha=(\alpha_k)_{k\in \mZ}$ be any sequence of even weights,
that is, the $\alpha_k$'s are positive real numbers satisfying
$$
\alpha_{-k}=\alpha_k \quad (k\in\mZ).
$$
Suppose, moreover that 
$$
\alpha_{k+l}\leq \alpha_k \alpha_l \quad (k,l\in\mZ).
$$
Consider the even-weighted Wiener algebra $W_\alpha(\mT)$ of the unit
circle $\mT$ given by
$$
W_\alpha(\mT)=
\Big\{f: f(z)=\sum_{k\in\mZ}f_k z^k \;(z\in \mT)
\textrm{ and } \sum_{k\in\mZ}\alpha_k |f_k|<+\infty\Big\},
$$
with pointwise operations, and the norm
$$
\|f\|_{W_\alpha(\mT)}
= 
\sum_{k\in\mZ}\alpha_k |f_k|,
\quad 
 f(z)=\sum_{k\in\mZ}f_k z^k \;(z\in \mT).
$$
Then this is a Banach algebra; see \cite[\S19.4,
p.118-120]{GelRaiShi}. The maximal ideal space of such even-weighted
Wiener algebras can be identified with the annulus 
$$
\mA(\rho)=\{z\in \mC:1/\rho \leq |z|\leq \rho\},
$$
where $\rho:= \displaystyle \inf_{k>0}
\sqrt[k]{\alpha_{k}}=\displaystyle\lim_{k\rightarrow\infty}
\sqrt[k]{\alpha_{k}}$.  The Gelfand transform is given by
$$
\widehat{f}(z)=\sum_{k\in\mZ}f_k z^k \quad (z\in\mA(\rho)).
$$
When $\rho= 1$ the weights given by $\alpha$ are said to satisfy the
{\em Gelfand-Raikov-Shilov condition} and the annulus $\mA(\rho)$
degenerates to the circle $\mT$. Examples that occur in the
applications are subexponential weights
$$
\alpha_k=e^{\alpha|k|^\beta}, \quad \alpha>0,\quad 0\leq \beta<1,
$$
and polynomial weights
$$
\alpha_k=(1+|k|)^s, \quad s\geq 0.
$$
Consider the following involution:
$$
f^\star(z)=\overline{f\left(\displaystyle \frac{1}{\overline{z}}\right)}
 \quad 
(z\in \mA(\rho),\; f\in W_\alpha(\mT)).
$$
Note that under the Gelfand-Raikov-Shilov condition the involution
$\cdot^{\star}$ reduces to the following:
$$
f^{\star}(z)= \overline{f(z)} \quad (z\in \mT, \;f\in W_\alpha(\mT)),
$$ 
and with this involution $W_\alpha(\mT)$ is a symmetric Banach
algebra.  Then for matrices $A,B,C$ with entries from $W_\alpha(\mT)$
the assumptions (A1)-(A3) of Theorem~\ref{main_theorem} are
automatically satisfied.  \hfill$\Diamond$
\end{example}

We remark that if the Gelfand-Raikov-Shilov condition is not
satisfied, then $W_\alpha(\mT)$ is not a symmetric algebra with the
involution considered in the previous example.

In the case of spatially invariant systems when the spatial argument
is continuous (rather than discrete), the state space is
$(L^2(\mR))^n$ (as opposed to $(L^2(\mT))^n$ or $(\ell^2(\mZ))^n$ in
the discrete case of the previous example), and the spatially decaying
property reduces to asking that the Riccati equation solution
belong to appropriate subalgebras of $(L^\infty (\mR))^{n\times n}$.
In this context, the following example is relevant.

\begin{example}[Even-weighted Wiener algebra of the line]
  Let the weight $w:\mR\rightarrow (0,+\infty)$ be a
  continuous function satisfying
$$
w(t+\tau)\leq w(t) w(\tau) \textrm{ and }w(-t)=w(t) \quad (t,\tau \in
\mR) .
$$
Let $L^1(\mR,w)$ be the set of all Lebesgue measurable complex valued
functions on $\mR$ such that
$$
\|f\|_{L^1(\mR,w)}=\int_\mR |f(t)|w(t)dt<+\infty.
$$
Then $L^(\mR,w)$ is a Banach algebra with this norm, with pointwise
addition and scalar multiplication, and with the multiplication
operation taken to be convolution:
$$
(f\ast g)(t)= \int_\mR f(t-\tau) g(\tau) d\tau\quad (t\in \mR).
$$
We adjoin a unit to the Banach algebra $L^1(\mR,w)$ to obtain its
unitization, denoted by $L^1(\mR,w)+\mC$ (see for instance
\cite[p.246]{Rud}), and the norm of an element $(f,\zeta)\in
L^1(\mR,w)+\mC$ is given by
$$
\|(f,\zeta)\|=\|f\|_{L^1(\mR,w)}+|\zeta|.
$$
It can be shown that (see \cite[(vii), p.816]{GohGolKaa}) 
$$
\tau:= \inf_{t<0} \frac{\log w(t)}{-t} =
\lim_{t\rightarrow -\infty} \frac{\log w(t)}{-t} \; \;(\geq 0)
$$
and that every complex homomorphism $\varphi$ on $L^1(\mR,w)+\mC$ is
either of one of the following types:
\begin{itemize}
\item[$\underline{1}^\circ$] $L^1(\mR,w)\subset \ker \varphi$. In this case, 
$$
\varphi((f,\zeta))=\varphi_\infty((f,\zeta))
:=\zeta \quad (f\in L^1(\mR,w),\;\zeta \in \mC).
$$
\item[$\underline{2}^\circ$] It is not the case that
  $L^1(\mR,w)\subset \ker \varphi$. In this case, $\varphi$ is a
  nontrivial multiplicative linear functional on $L^1(\mR;w)$, and so
  (see \cite[p.74]{Loo}), there exists a complex number $\lambda$ such
  that $-\tau \leq \textrm{Im}(\lambda)\leq \tau$ and
$$
\varphi((f,0))=\int_\mR f(t)e^{i\lambda t}dt \quad (f\in L^1(\mR,w)).
$$
Consequently,
$$
\quad\quad\quad\quad
\varphi((f,\zeta))=\varphi_\lambda((f,\zeta)):=\int_\mR
f(t)e^{i\lambda t}dt +\zeta \quad (f\in L^1(\mR,w),\;\zeta \in \mC).
$$
\end{itemize}
The Gelfand transform of $(f,\zeta)$ is given by 
$$
\widehat{(f,\zeta)}(\varphi)=\left\{\begin{array}{ll} 
\displaystyle \int_\mR f(t)e^{i\lambda t}dt +\zeta
& \textrm{if } \varphi=\varphi_\lambda,\;-\tau \leq
  \textrm{Im}(\lambda)\leq \tau,\\
\zeta & \textrm{if } \varphi=\varphi_\infty.
\phantom{\displaystyle \int_\mR f(t)e^{i\lambda t}dt}\end{array}
\right. 
$$
Consider the involution given by 
$$
(f,\zeta)^\star=(\overline{f(-\cdot)},\overline{\zeta}).
$$
The weight $w$ is said to satisfy the {\em Gelfand-Raikov-Shilov
  condition} if $\tau=0$. In this case, it is easy to check that 
$$
\widehat{(f,\zeta)^\star}(\varphi)= \overline{\widehat{(f,\zeta)}(\varphi)},
$$
for all $\varphi \in M(L^1(\mR,w)+\mC)$, $f\in L^1(\mR,w)$ and $\zeta
\in \mC$.  Thus $L^1(\mR,w)+\mC$ is a symmetric Banach algebra with
this involution. Hence for matrices $A,B,C$ with entries from
$L^1(\mR,w)+\mC$, the assumptions (A1)-(A3) of
Theorem~\ref{main_theorem} are then automatically satisfied.
\hfill$\Diamond$
\end{example}

\end{document}